\documentclass[hha]{article}
\usepackage{amsmath}
\usepackage{amssymb}
\usepackage{amsthm}
\usepackage[mathscr]{eucal}
\usepackage[all,knot,cmtip]{xy}
\xyoption{arc}

\numberwithin{equation}{section}
\newtheorem{definition}{Definition}[section]

\newtheorem{lem}[definition]{Lemma}

\newtheorem{thm}[definition]{Theorem}

\newtheorem{cor}[definition]{Corollary}

\begin{document}
\title{On loop spaces with marking}
\author{Ryo Horiuchi}
\date{}

\maketitle

\section{Introduction}
An analogous construction of simplicial homotopy group for Kan complex can be applied to saturated complicial sets\footnote{In \cite{Horiuchi}, the analogous construction is applied to non-saturated ones as well.} to give monoids.
In this paper, we investigate how the construction of loop spaces of Kan complexes lifts to the complicial setting and relates to such monoids.

Letting $\mathrm{msSet}_*$ denote the category of pointed simplicial sets with marking, we have the following.

\begin{thm}Loop space functor $\Omega:\mathrm{msSet}_*\to \mathrm{msSet}_*$ is a right Quillen functor from the model category of pointed $(\infty, n+1)$-categories to that of pointed $(\infty, n)$-categories.
\end{thm}
Here, the loop space functor is defined to be the right adjoint of reduced suspension functor, namely it is a straightforward generalization of the loop space functor of simplicial sets. 
We will study this in section 3.

Ozornova-Rovelli, Riehl and Verity (\cite{OR}, \cite{R}, \cite{V1}) constructed model structures on the category $\mathrm{msSet}$ of simplicial sets with marking, which are considered as simplicial models of weak higher categories.
More precisely, for $n\in\mathbb{N}$, $\mathrm{msSet}$ admits a model structure called the model of $(\infty, n)$-categories.
In particular, for $n=0, 1$, they are equivalent to the models of Kan complexes and quasicategories respectively (see \cite{V1} for more detail).
The category $\mathrm{msSet}$ also admits a model structure considered to give the model of weak $\omega$-categories.

By the proof of the theorem above, it follows that for any pointed saturated complicial set $X$, $\Omega(X)$ is also a pointed saturated complicial set since they are fibrant with respect to the model structure of  weak $\omega$-categories.
So its homotopy monoids make sense.

Our main theorem is the following.

\begin{thm}For any $n\in\mathbb{N}$ and any pointed saturated complicial set $X$, there is a monoid isomorphism $\tau_{n+1}(X)\cong\tau_{n}(\Omega(X))$.
\end{thm}

As the usual simplicial homotopy theory can be understood as the geometry of $(\infty, 0)$-categories, it may worth trying to study simplicial homotopy theory of $(\infty, n)$-categories for general $n$ to understand the geometry of spaces with non-reversible constituents.
Since both of the construction of homotopy monoids and that of loop spaces are straightforward generalizations of the classical ones, this theorem may be considered to be a natural generalization of the fundamental fact in the usual simplicial homotopy theory to the higher one.

In the present paper we heavily use the results in \cite{OR2} especially those about the suspension functor.


\section{Preliminaries}

Assuming the reader is familiar with simplicial sets, we recall some notations about simplicial sets with marking from \cite{OR}, \cite{OR2} and \cite{V1}.

\begin{definition}[\cite{V1}]A pair $(X, mX)$ is a simplicial set with marking\footnote{In \cite{V1}, simplicial sets with marking are called stratified simplicial sets} if
\begin{itemize}
  \item $X$ is a simplicial set,
  \item $mX$ is a set of simplices in $X$ such that $dX\subset mX$ and $X_0\cap mX=\emptyset$,
\end{itemize}
where $dX$ denotes the set of degenerate simplices in $X$.

A map of simplicial sets with marking $f:(X, mX)\to (Y, mY)$ is a simplicial map $f: X\to Y$ such that $f(x)\in mY$ for all $x\in mX$.
\end{definition}

We say a simplex of $(X, mX)$ is marked when it is an element of $mX$, and let $\mathrm{msSet}$ denote the category of simplicial sets with marking and their maps.

For simplicial sets with marking $(X, mX)$ and $(Y, mY)$, we say that $(X, mX)$ is a regular simplicial subset with marking of $(Y, mY)$ if $X\subset Y$ as simplicial sets and $mX=X\cap mY$ (cf. \cite{V1}).

For short, we often write $X$ for a simplicial set with marking $(X, mX)$ omitting $mX$.

\begin{definition}[\cite{OR2}, \cite{V1}]Let $n$ be a natural number and $k\in[n]$.
\begin{itemize}

  \item The standard thin $n$-simplex $\Delta[n]_t$ is the simplicial set with marking whose underlying simplicial set is the standard simplicial set $\Delta[n]$ and  
  \[
  m\Delta[n]_t = \begin{cases}
    d\Delta[n]\cup\{\operatorname{Id}_{[n]}\} & (n\neq 0) \\
    d\Delta[n] & (n=0).
  \end{cases}
\]
  
  \item The $k$-complicial $n$-simplex $\Delta^k[n]$ is the simplicial set with marking whose underlying simplicial set is the standard simplicial set $\Delta[n]$ and
  \[m\Delta^k[n]=d\Delta[n]\cup\{\alpha\in\Delta[n]| \{k-1, k, k+1\}\cap[n]\subset\operatorname{Im}(\alpha)\}.\] 
  
  \item The $n-1$-dimensional $k$-complicial horn $\Lambda^k[n]$ is the regular simplicial subset with marking of $\Delta^k[n]$ whose underlying simplicial set is the usual simplicial $k$-th horn.
  
  \item $\Delta^k[n]''$ (respectively $\Lambda^k[n]'$) is the simplicial set with marking whose underlying simplicial set is the same as that of $\Delta^k[n]$ (respectively $\Lambda^k[n]$) and its marked simplices are $m\Delta^k[n]$ (respectively $m\Lambda^k[n]$) with all its $n-1$-simplices.
  
  \item $\Delta^k[n]':=\Delta^k[n]\cup\Lambda^k[n]'.$
  \item $\Delta[3]_{eq}$ is the simplicial set with marking whose underlying simplicial set is the standard $3$-simplicial set and 
   \[m\Delta[3]_{eq}=d\Delta[3]\cup\Delta[3][2]\cup\Delta[3][3]\cup\{[02], [13]\}, \]
   where $[02]$ (respectively $[13]$) is the $1$-simplex whose image is $\{0, 2\}$ (respectively $\{1, 3\}$).
   \item $\Delta[3]^{\sharp}$ is the simplicial set with marking whose underlying simplicial set is the standard $3$-simplicial set with \[m\Delta[3]^{\sharp}=\bigcup_{n\geq1}\Delta[3][n].\]
  \end{itemize}
\end{definition}

For simplicial sets with marking $X$ and $Y$, $X\star Y$ denotes the join of $X$ and $Y$ (see for example \cite[Definition 2.4]{OR2}).
More precisely, the underlying simplicial set is the join of the underlying simplicial sets. So the set of its $n$-simplices is given by
\[(X\star Y)_n=\coprod_{\substack{k, l\geq-1\\ k+l=n-1}}X_k\times Y_l,\]
where $X_{-1}$ and $Y_{-1}$ are one point sets and a simplex $(x, y)\in X\star Y$ is marked if and only if $x\in X$ and $y\in Y$ are marked.

For a simplex $(x, y)\in X_k\times Y_l\subset(X\star Y)_{k+l+1}$, its face simplices are given by
\[
  d_i(x, y) = \begin{cases}
    (d^X_i(x), y) & 0\leq i\leq k \\
    (x, d^Y_{i-k-1}(y))  & k+1\leq i\leq k+1+l,
  \end{cases}
\]
and its degeneracy simplices are given by
\[
  s_i(x, y) = \begin{cases}
    (s^X_i(x), y) & 0\leq i\leq k \\
    (x, s^Y_{i-k-1}(y))  & k+1\leq i\leq k+1+l.
  \end{cases}
\]

The join operation is defined for augmented simplicial sets with marking and the empty one $\Delta[-1]$ is its unit.

\begin{definition}[\cite{V1}, \cite{OR}]\label{ext}Let $n$ be a natural number.
We call the following obvious maps $(\infty, n)$-elementary anodyne extensions:
\begin{itemize}
  \item the complicial horn extension map
  \[\Lambda^k[m]\to\Delta^k[m]\] for $m\geq 1$ and $k\in[m]$,
  
  \item the thinness extension map
  \[\Delta^k[m]'\to\Delta^k[m]''\] for $m\geq 2$ and $k\in[m]$,
  \item  the triviality extension map 
  \[\Delta[m]\to\Delta[m]_t\] for $m\geq n+1$,
  \item the saturation extension map
  \[\Delta[m]\star\Delta[3]_{eq}\to\Delta[m]\star\Delta[3]^{\sharp}\]
  for $m\geq-1$.

\end{itemize}
\end{definition}

Note that only the triviality extension maps depend on $n$. We call complicial horn extension maps, thinness extension maps and saturated extension maps from Definition \ref{ext} {\it elementary anodyne extensions}.

\begin{definition}[\cite{V1}, \cite{OR}] Let $n$ be a natural number. A simplicial set with marking $X$ is an $n$-trivial saturated complicial set (respectively a saturated complicial set) if it has the right lifting property with respect to $(\infty, n)$-elementary anodyne extensions (respectively elementary anodyne extensions).

A map $f:X\to Y$ of simplicial sets with marking is an $(\infty, n)$-weak equivalence (respectively a saturated complicial weak equivalence) if for any $n$-trivial saturated complicial set (respectively any saturated complicial set) $Z$ the map 
\[f^*: Z^Y\to Z^X\]
on internal homs is a homotopy equivalence in the sense of \cite{V1}.
\end{definition}

Let $\mathrm{msSet}_*$ denote the category of pointed simplicial sets with marking and pointed maps.

\begin{thm}[\cite{OR2}]\label{ORmodel}For any $n\in\mathbb{N}$, $\mathrm{msSet}_*$ admits the following model structure, where 
\begin{itemize}
\item the cofibrations are precisely monomorphisms and 
\item weak equivalences are precisely those maps whose underlying maps of simplicial sets with marking are $(\infty, n)$-weak equivalences (respectively saturated complicial weak equivalences).
\end{itemize}
The fibrant objects are precisely the pointed simplicial sets with marking whose underlying simplicial sets with marking are $n$-trivial saturated complicial sets (respectively saturated complicial sets).

\end{thm}

For $n\in\mathbb{N}$, we call this model structure {\it the model structure of pointed $(\infty, n)$-categories}. Note that $(\infty, n)$-elementary anodyne extensions are weak equivalences in this model structure.


\section{Results}

Using the results and arguments in \cite{OR2} about the suspension functor of simplicial sets with marking, we consider the reduced one in the following.

For a pointed simplicial set with marking $(X, x_\top)$, $\Delta[0]\star X$ has two $0$-cells $x_\bot$ and $x_\top$.
The former is the one corresponding to the $0$-simplex of $\Delta[0]$ and the latter is that corresponding to  the base point of $X$.
We take $x_\bot$ as the base point of $\Delta[0]\star X$.

Also $\Delta[0]\star X$ has the $1$-simplex $x_\bot x_\top$ whose source is $x_\bot$ and target is $x_\top$. 
Note that this $1$-simplex is not marked.
We let $\langle x_\bot x_\top\rangle$ denote the pointed simplicial subset with marking of $\Delta[0]\star X$ generated by $x_\bot x_\top$.

\begin{definition}[reduced suspension]
We define the reduced suspension functor\footnote{In \cite{OR2}, to define the suspension, $X\star\Delta[0]$ is used instead of $\Delta[0]\star X$. It may be more reasonable to call $\Sigma_+$ the reduced {\it left} suspension functor. If we take the right one, the bijection in Theorem \ref{main} may not be a homomorphism of monoids.} $\Sigma_+:\mathrm{msSet}_*\to \mathrm{msSet}_*$ by the following pushout
 \[
   \xymatrix{
 \Delta[-1]\star X\cup\langle x_\bot x_\top\rangle\ar[d]\ar[r]& \Delta[0]\ar[d]\\
 \Delta[0]\star X\ar[r]&\Sigma_+X,\\     
}
\]
for an object $X\in\mathrm{msSet}_*$, where the base point of $\Sigma_+X$ is induced by $x_\bot$. For a map $f$ of pointed simplicial sets with marking, $\Sigma_+(f)$ is the induced one.
\end{definition}

As shown in \cite{OR2} the suspension functor is left adjoint. The reduced one also has the right adjoint.

\begin{lem}$\Sigma_+$ is a left adjoint functor.
\end{lem}
\begin{proof}We construct the right adjoint functor of $\Sigma_+$ by using the right adjoint functor $\operatorname{P}^{\triangleleft}$ of $\Delta[0]\star(-)$ studied in \cite{V1}.

Let $(Z, a)$ be a pointed simplicial set with marking. Then we have a simplicial set with marking $\Omega(Z, a)$ defined by the following pullback
 \[
   \xymatrix{
\Omega(Z, a)\ar[d]\ar[r]& \operatorname{P}^{\triangleleft}_a(Z)\ar[d]\\
 \Delta[0]\ar[r]_a&Z,\\     
}
\]
where the right-hand side vertical map is induced by the natural inclusions $\Delta[n]\hookrightarrow\Delta[0]\star\Delta[n]$.
We give $\Omega(Z, a)$ the base point, which we write $a$ again, by the constant map on $a$ $\Delta[0]\to\operatorname{P}^{\triangleleft}_a(Z)$ and identity map $\Delta[0]\to\Delta[0]$.

To show that this gives a right adjoint of $\Sigma_+$, assume that we have a pointed map $\Sigma_+(X)\to Z$ with respect to $a\in Z$. Then we have the following diagram due to the adjunction $\Delta[0]\star(-)\dashv\operatorname{P}^{\triangleleft}$
 \[
   \xymatrix{
   \Delta[0]\ar@/^11pt/[rrd] \ar@/_14pt/[rdd]\ar[rd]^{x_{\top}}&&\\
&X\ar[d]\ar[r]& \operatorname{P}^{\triangleleft}_a(Z)\ar[d]\\
 &\Delta[0]\ar[r]_a&Z.\\     
}
\]
This diagram gives the desired map $X\to\Omega(Z, a)$, which defines a right adjoint functor $\Omega$ of $\Sigma_+$.
\end{proof}

We may write $\Omega(Z)$ instead of $\Omega(Z, a)$ for short. By definition, an $n$-simplex $\Delta[n]\to\Omega(Z)$ gives rise to a map $\Delta[n]\to\Delta[0]\xrightarrow{a} Z$. Thus the $0$-th face of the corresponding $n+1$-simplex in $Z$ is degenerated on $a$.

This functor is compatible with the model structures in the following sense.

\begin{lem}The reduced suspensionfunctor $\Sigma_+$ is left Quillen from the model structure of pointed $(\infty, n)$-categories to that of pointed $(\infty, n+1)$-categories.
\end{lem}
\begin{proof}It is clear that $\Sigma_+$ preserves cofibrations.
Let $f:X\to Y$ be a pointed $(\infty, n+1)$-weak equivalence.
Since $\Delta[-1]$ is the unit for join, we have a pointed $(\infty, n+1)$-weak equivalence.
\[f_*: \Delta[-1]\star X\cup\langle x_\bot x_\top\rangle\to \Delta[-1]\star Y\cup\langle y_\bot y_\top\rangle,\] which fits into the following commutative diagram
 \[
   \xymatrix{
 \Delta[0]\ar[d]_{=}&\ar[l]\Delta[-1]\star X\cup\langle x_\bot x_\top\rangle\ar[d]_{f_*}\ar[r]& \ar[d]^{\Delta[0]\star f} \Delta[0]\star X\\
 \Delta[0]&\ar[l]\Delta[-1]\star Y\cup\langle y_\bot y_\top\rangle\ar[r]&  \Delta[0]\star Y\\     
}
\]
By \cite[Proposition 2.5]{OR2}, the right-hand side vertical map is a pointed $(\infty, n+1)$-weak equivalence since every object is cofibrant.
The right-hand side horizontal map is a cofibration.
Thus, we obtain the pointed $(\infty, n+1)$-weak equivalence $\Sigma_+f:\Sigma_+X\to\Sigma_+Y$ as the push out of the diagram.
This shows that $\Sigma_+$ sends all $(\infty, n)$-elementary anodyne extensions except for the triviality extension maps to $(\infty, n+1)$weak equivalences.

By \cite[Lemma 1.8]{OR2}, it is enough to show that the map
\[\Sigma_+\Delta[n+1]\to\Sigma_+\Delta[n+1]_t\]
is an $(\infty, n+1)$-weak equivalence.
By construction, the underlying simplicial sets of $\Sigma_+\Delta[n+1]$ and $\Sigma_+\Delta[n+1]_t$ are the same, and the underlying map of simplicial sets is the identity map.
The only difference is that the non-degenerate $(n+2)$-simplex of $\Sigma_+\Delta[n+1]_t$ is marked.
Therefore we have the following pushout
 \[
   \xymatrix{
 \Delta[n+2]\ar[d] \ar[r]&\Delta[n+2]_t\ar[d]\\
 \Sigma_+\Delta[n+1]\ar[r]&\Sigma_+\Delta[n+1]_t.\\     
}
\]
Since the upper horizontal map is an acyclic cofibration, the lower map is an $(\infty, n+1)$-weak equivalence as desired.
\end{proof}

By Theorem \ref{ORmodel}, $k$-trivial complicial sets are fibrant with respect to the model structure of $(\infty, k)$-categories. We get the following.

\begin{cor}The loop space functor $\Omega$ is right Quillen from the model structure of pointed $(\infty, n+1)$-categories to that of pointed $(\infty, n)$-categories.

In particular for any pointed $n+1$-trivial saturated complicial set $X$, $\Omega(X)$ is a pointed $n$-trivial saturated complicial set.
\end{cor}

By the same arguments, $\Omega$ is right Quillen with respect to the model structure for pointed saturated complicial sets. Therefore for any pointed saturated complicial set $X$, so is $\Omega(X)$.

As is well known, for a pointed Kan complex $A$ and a natural number $n$, there is an isomorphism $\pi_{n+1}(A)\cong\pi_{n}(\Omega(A))$ between their homotopy groups, where $\Omega$ denotes the usual loop space functor.
In the following, we show that the analogous result for saturated complicial sets holds.
As is shown in \cite{Horiuchi}, the construction of simplicial homotopy groups (see for example \cite{GJ} or \cite{M}) lifts to the complicial setting in a straightforward way.

For a pointed simplicial set with marking $(X, x)$, we let $X^x_{n}$ denote the set of $n$-simplices in $X$ whose boundaries are the degenerate simplicies on $x$.
Then there is an evident map $\psi:X^x_{n+1}\to \Omega(X)^x_{n}, \alpha\mapsto\alpha$.
By definition, for any $\alpha\in X^x_{n+1}$, its $0$-th face is the degenerated simplex on $x$. So $\psi(\alpha)=\alpha$ is indeed an element of $\Omega(X)^x_{n}$.
Note that this map sends marked simplicies to marked simplicies.

\begin{thm}\label{main}For any pointed saturated complicial set $X$ and $n\geq 1$, there is a monoid isomorphism $\tau_{n+1}(X)\cong\tau_{n}(\Omega(X))$.
\end{thm}
\begin{proof} 
The map induces a bijection $\Psi:\tau_{n+1}(X)\to\tau_{n}(\Omega(X))$, which is a homomorphism of monoids. Indeed, by construction of homotopy monoids, for $\alpha, \beta\in X^x_{n+1}$, the multiplication $[\alpha][\beta]$ in $\tau_{n+1}(X)$ of their homotopy classes can be written as a homotopy class $[d^X_{n+1}(\theta)]$, where $\theta$ is an $n+2$-simplex obtained by $\alpha$ and $\beta$.
By the definition of the simplicial structure of the join, the $n+1$-st face of an $n+2$-simplex in $X$ coincides with the $n$-th face of the corresponding $n+1$-simplex in $\Omega(X)$.
So $\Psi([\alpha][\beta])=\Psi([d^X_{n+1}(\theta)])=[d^{\Omega(X)}_{n}(\theta)]=[\alpha][\beta]=\Psi([\alpha])\Psi([\beta])$.
By definition, $\Psi$ takes the unit to the unit.
\end{proof}

By the same argument, there is a bijection $\tau_1(X)\cong\tau_0(\Omega(X))$ for any pointed saturated complicial set $X$. Thus we can give $\tau_0(\Omega(X))$ the monoid structure via this bijection.

Considering this theorem, we may define the spheres in the complicial setting as follows.
First the $0$-dimensional sphere $\operatorname{S}^0$ is the simplicial set with marking consisting of two points.
For any positive natural number $n$, the $n$-dimensional sphere $\operatorname{S}^n$ is defined to be $\Sigma_+(\operatorname{S}^{n-1})$.
Note that $\operatorname{S}^n$ is not isomorphic to the $n$-times smash of $\operatorname{S}^1$ in general.

By the construction of $\tau_*$ and the definition of triviality, for any $n\geq0$, $k\geq1$ and any pointed $n$-trivial saturated complicial set $X$, $\tau_{n+k}(X)$ is a group. This observation is compatible with the theorem above. In other words, for such an $X$, $\Omega^n(X)$ is a $0$-trivial saturated complicial set. By construction again, $\tau_k(\Omega^n(X))$ is isomorphic to the $k$-th simplicial homotopy group of the Kan complex associated to $\Omega^n(X)$.

As a classical fact, for a pointed Kan complex $A$, $\pi_{k+1}(A)$ is an abelian group. Hence by the same argument above, for a pointed $n$-trivial saturated complicial set $X$, the monoid $\tau_{n+k+1}(X)$ is an abelian group.

\section*{Acknowledgments}
This work is a part of a project suggested by Lars Hesselholt. I appreciate him guiding myself to it.

\end{document}